\documentclass[11pt,letterpaper]{amsart}%
\usepackage{amsmath}
\usepackage{amsfonts}
\usepackage{amssymb}
\usepackage{graphicx}
\usepackage{amsmath}
\usepackage{txfonts}
\usepackage{graphicx}
\usepackage{amsthm}
\usepackage{amsfonts}
\usepackage{amssymb}
\usepackage[utf8]{inputenc}
\usepackage{xcolor}
\usepackage{ifthen}%
\providecommand{\U}[1]{\protect\rule{.1in}{.1in}}
\providecommand{\U}[1]{\protect\rule{.1in}{.1in}}
\newif\ifcomments
\commentsfalse
\newboolean{bluecomments}
\setboolean{bluecomments}{true}
\newcommand{\bluecomment}[1]{\ifthenelse{\boolean{bluecomments}}{\textcolor{blue}{#1}}{}}
\newcounter{mycounter}
\numberwithin{mycounter}{section}
\newtheorem{theorem}[mycounter]{Theorem}
\theoremstyle{plain}

\newtheorem{claim}[mycounter]{Claim}

\newtheorem{corollary}[mycounter]{Corollary}

\newtheorem{lemma}[mycounter]{Lemma}

\newtheorem{proposition}[mycounter]{Proposition}

\numberwithin{equation}{section}
\theoremstyle{definition}

\begin{document}
\title[Grim Raindrop]{Grim Raindrop: a translating solution to curve diffusion flow }
\author{W. Jacob Ogden}
\address{Department of Mathematics\\
University of Washington, Seattle, WA 98105}
\email{wjogden@uw.edu}
\author{Micah Warren}
\address{Department of Mathematics\\
University of Oregon, Eugene, OR 97403}
\email{micahw@uoregon.edu}

\begin{abstract}
We show the existence of a properly immersed translating solution to curve
diffusion flow in the plane. Curve diffusion flow is a higher order version of
curve shortening flow, namely
\[
\left(  \frac{dX}{dt}\right)  ^{\perp}=-\kappa_{ss}N.
\]

\end{abstract}
\maketitle

\ 

\section{Introduction}

In this note we present a nontrivial properly immersed translating solution to
the curve diffusion flow. Curve diffusion flow is a fourth order analogue to curve shortening flow. For closed curves, it is a gradient flow of arc length among curves enclosing the same signed area. 
This flow has a long history of interest in a wide variety of fields in both
pure and applied mathematics, going back to Mullins \cite{Mullins1957} in 1957; see \cite{figure8} for a more complete list of references. In
higher dimensions, the curve diffusion flow generalizes to the surface
diffusion flow for hypersurfaces \cite{EMS}, and to the gradient flow for
volume of Lagrangian submanifolds within a given Hamiltonian isotopy class
\cite{CW24}. 

In \cite{Polden} it was shown that singularities occur for the flow. In
particular, an immersed figure-8 curve must develop singularities under the
flow. Other examples of curves which develop singularities are provided in \cite{EscherIto}. It is a natural question to consider self-similar solutions, as these
may help model the formation of singularities.

In \cite{figure8} it was shown that:

\begin{itemize}
\item (Corollary 4) The only open, properly immersed stationary solutions to
the flow are lines;

\item There are many stationary solutions that are not properly immersed (Euler spirals);

\item (Proposition 11) If an open translator has curvature vanishing at
$+\infty$ and $-\infty$ and the angle the translating vector makes with the
tangent is the same at both $+\infty$ and $-\infty$, then the translator is a
stationary line;

\item (Proposition 12) If $X: \mathbb{R} \to\mathbb{R}^{2}$ is an open
properly immersed translator satisfying
\[
\lim_{\rho\to\infty} \frac{ 1}{ \rho^{2}} \int_{X^{-1} (B_{\rho}) }\kappa^{2}
\: ds + \frac1 \rho\int_{X^{-1} ( B_{2\rho} ) } | \langle\vec E , T \rangle|
\: ds =0,
\]
where $\vec E$ is the translating velocity and $T$ is the oriented unit
tangent vector, then $X(\mathbb{R})$ is a straight line.
\end{itemize}

The authors asked whether the growth condition or properness are necessary in
Proposition 12. Our main result partially answers this question by showing
that properness by itself does not imply the translator is a line.

Our main result is:

\begin{theorem}
\label{theorem} There exists a properly immersed translating solution to the
curve diffusion flow with nonconstant curvature.
\end{theorem}

Our strategy is to use a shooting method together with symmetry to show that a
bounded yet nontrivial solution to a certain third order nonlinear ODE
(\ref{ODE}) exists. The solution of the ODE will describe the angle of the
tangent vector of the profile curve for the translating solution of the flow.

\section{Curve diffusion flow}

For a one dimensional manifold $M$, a map
\[
X(\omega,t) : M\times I \to\mathbb{R}^{2}%
\]
which satisfies
\begin{equation}
\left(  \frac{dX}{dt}\right)  ^{\perp}=-\kappa_{ss}N \label{equation1}%
\end{equation}
is a solution to the curve diffusion flow. Here $\kappa$ is the curvature,
which can be defined as the derivative of the angle $\theta$ that the unit
tangent vector makes with the $x$-axis with respect to the arc length
parameter $s$. The unit normal vector $N$ is the complex rotation of the
oriented unit tangent vector of the curve.

For compact $M$ it is easy to check using the first variational formula that
this is the gradient flow for arc length on the space of curves with the same
enclosed area, where the tangent space consists of variations of the form
\[
\frac{dX}{dt}=J\nabla f
\]
for $f$ a function on the curve $\gamma$, where $J$ is the standard complex
rotation and the tangent space is equipped with the $L^{2}$ metric
\[
\left\Vert f\right\Vert ^{2}=\int_{\gamma}f^{2}ds.
\]

\subsection{Translating solutions}

Suppose that the unprojected flow satisfies
\[
\frac{dX}{dt}=\vec{E}%
\]
for some fixed unit vector $\vec{E}.$ From (\ref{equation1}) we have
\[
\vec{E}\cdot N=-\kappa_{ss}.
\]
Now, without loss of generality, let
\[
\vec{E}=\left(  0,1\right)  =J(1,0)
\]
so that
\begin{align*}
\vec{E}\cdot N  &  =J(1,0)\cdot JT\\
&  =(1,0)\cdot T\\
&  =\cos\theta,
\end{align*}
and the equation becomes
\[
\cos\theta=-\kappa_{ss}%
\]
or
\[
\cos\theta=-\theta_{sss},%
\]
which we write as a third order nonlinear ODE
\begin{equation}
\theta_{sss}=-\cos\theta. \label{ODE}%
\end{equation}

An arc-length parameterized path
\[
\gamma:\mathbb{R\rightarrow R}^{2}%
\]
whose angle $\theta$ satisfies \eqref{ODE} is the profile of a
translating solution to the curve diffusion flow. If, in addition
\begin{align*}
\theta &  \rightarrow\frac{\pi}{2}\text{ as }s \to\infty,\\
\theta &  \rightarrow-\frac{\pi}{2}\text{ as }s \to-\infty
\end{align*}
then the solution will be properly immersed and have nonconstant curvature.

Before constructing a solution with the desired properties, we give a simple criterion which provides a large family of solutions to \eqref{ODE} which correspond to translators which are not properly immersed, as their curvatures grow as asymptotically linear functions of arc length (which is also the case for the stationary Euler spiral solutions to the curve diffusion flow). 

\begin{proposition} Suppose that $\theta$ satisfies \eqref{ODE} with $\theta_s (0) \geq a >0$ and $\theta_{ss} (0 ) > \frac 2 a .$ Then the corresponding curve $\gamma$ is not properly immersed. 
\end{proposition} 

\begin{proof} We claim that $\theta_{ss} $ is uniformly bounded away from 0 as $s \to \infty$. First observe that $\theta_{ss}$ can never reach $0$. For contradiction, assume that $s_0>0$ is the first time such that $\theta_{ss} (s_0) =0$. Then for any $s \in [0, s_0]$, \begin{align} 
|\theta_{ss} (s) - \theta_{ss} (0) | &= \left | \int _0^s  \cos ( \theta(\sigma)  ) \: d\sigma \right |  \label{integralestimate}\\
& = \left | \int_0^s \frac{1}{ \theta_s ( \sigma )} \frac{d}{d \sigma } \sin ( \theta ( \sigma )) \: d \sigma \right |. \notag
\end{align}

Integrating by parts, 
\begin{align*} 
\left | \int_0^s \frac{1}{ \theta_s ( \sigma )} \frac{d}{d \sigma } \sin ( \theta ( \sigma )) \: d \sigma \right | 
&  \leq  \left |  \frac{ 1 }{ \theta_s ( s) } \sin(\theta (s))- \frac{1}{\theta_s(0)} \sin ( \theta (0) ) \right | 
\\ & \quad + \left | \int_0^{s} \frac{ d}{ d \sigma } \frac{ 1}{ \theta _s( \sigma ) } \sin ( \theta ( \sigma ) )  \: d \sigma \right | \\
& \leq \frac{ 1} { \theta _s ( 0) } + \frac{1}{ \theta_s(s) }  + \left | \int_0^s \frac{ d}{d\sigma} \frac{ 1}{ \theta _s (\sigma) } \: d \sigma \right | \\
& = \frac{ 1} { \theta _s ( 0) } + \frac{1}{ \theta_s(s) }  + \left|  \frac{ 1} { \theta _s ( s) } - \frac{1}{ \theta_s(0) } \right |   \\
& = \frac 2 a,
\end{align*}  
where we used the fact that $\theta_s $ is monotone increasing. 
This is a contradiction since $\theta_{ss} (0) > \frac 2 a,$ so $\theta_{ss} >0$ for all $s>0$. 
Considering \eqref{integralestimate} again now shows that 
$$ \theta_{ss} (s)\geq \theta_{ss} (0) - \frac 2 a >0$$
for all $s>0$, so $\theta_{ss}$ is uniformly bounded below by a positive constant $\varepsilon.$ Then 
$$ \kappa ( s) - \kappa (0) = \int_0^s \theta_{ss} \: ds \geq \varepsilon s, $$
so the curvature of $\gamma$ grows at least linearly as $s \to \infty$. 
(Note also that  
$$  \kappa ( s) - \kappa (0) \leq\left ( \theta_{ss}(0) + \frac 2a\right  )s, $$
so the curvature also has a linear upper bound.) 
Since the curvature of $\gamma$ is increasing, the piece $\{ \gamma (s) \: : \: s > 0 \} $ is contained in the osculating circle at $0$. Since $\gamma$ has infinite arc length inside a compact set, it is not properly immersed. 
\end{proof}

\subsection{Strategy of construction}

Our strategy for constructing a properly immersed translating solution is:

\begin{enumerate}
\item Modify the ODE \eqref{ODE} to obtain a monotone ODE \eqref{ODE2}. Show that any solution to the modified equation which stays bounded for $s>0$ must converge to $\pi/2 $ as $s \to \infty$.

\item Show that any solution to the modified equation with suitable initial conditions \eqref{ICs} which stays bounded for $s> 0$ is a solution to the original equation \eqref{ODE}. 

\item Demonstrate the existence of a bounded solution $\theta(s)$ to the modified equation. 

\item Reflect the solution in Step 3 across $s=0$ to get an entire solution which
also converges to $-\pi/2$ as $s\rightarrow-\infty.$

\item With $\theta$ determined as a function of arc length, construct the path $\gamma$ forwards and backwards. The tangent directions
approach $\pm\pi/2$ so the solution must be properly immersed.
\end{enumerate}

\begin{figure}[h]
\centering
\includegraphics[width=.3\textwidth]{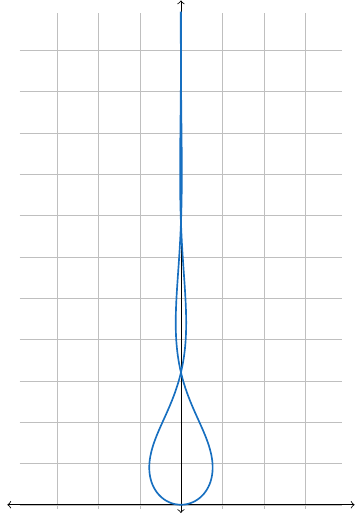} \caption{Grim raindrop.
It's grim because it falls upwards.}%
\label{fig:gr}%
\end{figure}

\section{Proof of Theorem \ref{theorem}}

\subsection{Step 1}

We start by modifying the ODE \eqref{ODE} slightly so that the right-hand side
is monotone: Consider the ODE%
\begin{equation}
u^{\prime\prime\prime}=f(u) \label{ODE2}%
\end{equation}
with \
\[
f(u)=%
\begin{cases}
-\cos u & \text{ if }0\leq u\leq\pi\\
-1 & \text{ if }u\leq0\\
+1 & \text{ if }u\geq\pi.
\end{cases}
\]
It is more straightforward to apply shooting methods to monotone ODE. If we
have a solution whose values lie within the range $\left[  0,\pi\right]  $
this will also be a solution of the original ODE \eqref{ODE}. Our goal is then
to shoot to find a solution to the modified ODE, and then argue that this
solution satisfies the original ODE.

The following claim sets up a shooting method.

\begin{lemma}
\label{order} Let $u_{a}$ be the solution to the ODE \eqref{ODE2}
\[
u^{\prime\prime\prime}=f(u)
\]
with initial conditions
\begin{align}
u(0)  &  =0,\label{ICs}\\
u^{\prime}(0)  &  =a,\nonumber\\
u^{\prime\prime}(0)  &  =0.\nonumber
\end{align}
These have the following order preservation property: If
$
a<b
$
then
\[
u_{a}(s)<u_{b}(s)\text{ for all }s>0.
\]

\end{lemma}

\begin{proof}
Let
\[
v(s)=u_{b}(s)-u_{a}(s).
\]
Then
\begin{align*}
v(0)  &  =0,\\
v^{\prime}(0)  &  =b-a,\\
v^{\prime\prime}(0)  &  =0,\\
v^{\prime\prime\prime}(0)  &  =0.
\end{align*}
So $v>0,$ and $v^{\prime}>0$ for small values of $s.$ Let $s_{0}$ be the first
value such that $v^{\prime}(s_{0})=0.$ Then
\[
0>\int_{0}^{s_{0}}v^{\prime\prime}(\tau)\:d\tau=\int_{0}^{s_{0}}\int_{0}%
^{\tau}\left(  f(u_{b})-f(u_{a})\right)  \: d\sigma\: d\tau
\]
but this is a contradiction, because on the whole interval we had
$f(u_{b})-f(u_{a})\geq0.$ There is no such $s_{0};$ we may conclude
$v^{\prime}>0$ and hence $u_{b}(s)>u_{a}(s)$ for all postive $s.$
\end{proof}
The following is a straightforward application of the fundamental theorem of
calculus. Because we refer to this argument frequently in the sequel we give
it a name.  
\begin{lemma}\label{trifecta}
(Trifecta).  Suppose that $v$ is a solution to a monotone third order ODE,
that is $v^{\prime\prime\prime}=h(v)$ for some monotone non-decreasing
function $h(v).$  If at any point $s$, $h(v(s)),v^{\prime}(s)$ and $v^{\prime\prime}(s)$
are all nonnegative and at least one of these is positive, the solution $v$
will increase forever without reaching a future critical point, and $s$ cannot be a local maximum itself.
\end{lemma}

 We now begin analyzing the behavior of bounded solutions to \eqref{ODE2}. 

For cleaner presentation, let $v=u-\pi/2,$ and consider instead the following
ODE:
\begin{equation}
v^{\prime\prime\prime}=g(v) \label{ODEsin}%
\end{equation}
with \
\[
g(v)=%
\begin{cases}
\sin v & \text{ if }-\pi/2\leq v\leq\pi/2\\
-1 & \text{ if }v\leq-\pi/2\\
+1 & \text{ if }v\geq\pi/2.
\end{cases}
\]

\begin{lemma}
\label{strictmin}Suppose that $v$ is a solution to \eqref{ODEsin}. Suppose
that at $s_{\max}$ the solution $v$ has a positive local maximum and at
$s_{\min},$ the next critical point, $v$ has a negative local minimum. Then
$v(s_{\max})>\left\vert v(s_{\min})\right\vert $ .
\end{lemma}

\begin{proof}
We know $v^{\prime}(s_{\max})=0~$and $v^{\prime\prime}(s_{\max})<0;$ if
$v^{\prime\prime}(s_{\max})=0$ the trifecta (Lemma \ref{trifecta}) would preclude further critical points. Similarly, we have $v^{\prime\prime}(s_{\min
})> 0$.   There must be a point $s_{0}
\in(s_{\max},s_{\min})$ where $v^{\prime\prime}(s_{0})=0.$ To avoid the trifecta we must have $v(s_{0}%
)\geq0$.
Thus
\begin{align*}
0  &  \leq v^{\prime\prime}(s_{\min})-v^{\prime\prime}(s_{0})=\int_{s_{0}%
}^{s_{\min}}v^{\prime\prime\prime}(s)\: ds\\
&  =\int_{s_{0}}^{s_{\min}}g( v(s))\:ds=\int_{v(s_{0})}^{v(s_{\min})}g(
v)\frac{ds}{dv}\: dv\\
&  =\int_{v(s_{\min})}^{v(s_{0})}g( v)\left(  -\frac{ds}{dv}\right)  \:dv \label{f67}
\end{align*}
using a change of variables $s\mapsto v(s).$

Now let
\[
\rho(v)=\left(  -\frac{ds}{dv}\right)  =-\frac{1}{v^{\prime}(s)}.
\]
Compute%
\begin{align*}
\frac{d}{dv}\rho(v) &  =\frac{d}{ds}\left(  \frac{-1}{v^{\prime}(s)}\right)
\frac{ds}{dv}\\
&  =\frac{v^{\prime\prime}}{(v^{\prime})^{3}}\leq0\text{ on }[s_{0},s_{\min})
\end{align*}
and strictly negative on the interior. Thus, as a function of $v$, $\rho$ is
decreasing and we have for any $0<w<\min\{v(s_{0}),|v(s_{min})|\}$
\begin{equation}
\rho(w)<\rho(-w).\label{strict}%
\end{equation}
Now for purposes of contradiction assume that%
\[
v(s_{0})\leq\left\vert v(s_{\min})\right\vert .
\]
Then
\begin{align*}
0 &  \leq\int_{v(s_{\min})}^{v(s_{0})}g( v)\rho(v)\:dv\\
&  =\int_{v(s_{\min})}^{-v(s_{0})}g( v)\rho(v)\:dv+\int_{-v(s_{0})}^{v(s_{0}%
)}g( v)\rho(v)\:dv
\end{align*}
with the first term nonpositive, so we have
\begin{align*}
0 &  \leq\int_{-v(s_{0})}^{v(s_{0})}g( v)\rho(v)\:dv\\
&  =\int_{0}^{v(s_{0})}g(v)\rho(v)\:dv+\int_{-v(s_{0})}^{0}g( v)\rho\left(
v\right)  \:dv\\
&  =\int_{0}^{v(s_{0})}g( v)\rho(v)\:dv+\int_{v(s_{0})}^{0}g\left(
-v\right)  \rho\left(  -v\right)  \:d(-v)\\
&  =\int_{0}^{v(s_{0})}g( v)\left[  \rho(v)-\rho\left(  -v\right)  \right]\:
dv<0
\end{align*}
using the fact that $g$ is odd, and
with the last inequality following from \eqref{strict}. This contradicts
$v(s_{0})\leq\left\vert v(s_{\min})\right\vert .$ Thus $v\left(  s_{0}\right)
>\left\vert v(s_{\min})\right\vert $ which implies that $v(s_{\max})>v\left(
s_{0}\right)  >\left\vert v(s_{\min})\right\vert $ and proves the Lemma.
\end{proof}

Noting that an order-reversed proof holds for a maximum subsequent to a
minimum, it follows that if a solution bounces between positive local maxima
and negative local minima, the distance between these must strictly be
decreasing. We would like to show it decreases all the way to zero.

Next we quantify the difference between $v(s_{\max})$ and $v\left(
s_{0}\right)  $ in the above proof to get a more precise decay estimate.  

\begin{lemma}
\label{strictbounce}Suppose that $v$ is a solution to \eqref{ODEsin}. Suppose
that $v(s_{\max} ) \leq  \pi $ is a positive local maximum and at $s_{\min},$ the next critical point, $v$ has a
negative local minimum. Then
\[
\left\vert v(s_{\min})\right\vert <\delta(v(s_{\max}))v(s_{\max})
\]
where
\[
\delta(z)=1-\frac{z^{2}}{192}<1.
\]
If $v(s_{\max} ) > \pi$, and $v$ has a negative local minimum at $s_{\min} $, then 
$$ | v ( s _ {\min} ) | < v(s_{\max} ) - \frac{ \pi^3 }{ 192} .$$

\end{lemma}

\begin{proof}
Consider solutions which begin from a positive local maximum at $s=0$; $v$ is
the solution of the initial value problem with
\begin{align*}
v(0) &  =v(s_{\max})=k>0,\\
v^{\prime}(0) &  =0,\\
v^{\prime\prime}(0) &  =-b<0.
\end{align*}
By reasoning in the previous proof, there is some $s_0>0$ such that $v^{\prime
\prime}(s_{0})=0$ and $v(s_{0})>0.$ 

Suppose that $v(s_{\max} ) \leq \pi.$ 
First we claim that the hypothesis that a negative minimum occurs subsequently implies that $b>\frac{k}{4}$: Assume not. It follows that
\[
v(s)\geq k-\frac{k}{4}\frac{s^{2}}{2}\geq k\left (1-\frac{s^{2}}{8}\right )
\]
while $v$ remains positive. So $v(s)\geq k/2$ while $s\in[0,2]$ and
we estimate  \
\begin{align*}
v^{\prime}(2) &  =\int_{0}^{2}v^{\prime\prime}(s)\:ds\\
&  =\int_{0}^{2}\left(  -b+\int_{0}^{s}v^{\prime\prime\prime}(t)\:dt\right)
\:ds\\
&  \geq-2\frac{k}{4}+2\sin\left(  \frac{k}{2}\right)  >0\end{align*}
as $2 \sin x > x $ for $x \in (0, \pi/2).$ 
Thus $v$ must have achieved a local minimum while its value was positive, and Lemma \ref{trifecta} then contradicts the hypothesis. Thus $b>\frac{k}{4}.$  

Next we claim that $s_{0}\geq b:$
\begin{align*}
b &  =v^{\prime\prime}(s_{0})-v^{\prime\prime}(0)\\
&  =\int_{0}^{s_{0}}v^{\prime\prime\prime}(s)\:ds\\
&  \leq\int_{0}^{s_{0}}\:ds=s_{0}.
\end{align*}

Now since $v^{\prime\prime\prime}\leq1$, we have
\[
v(s)\leq k-\frac{b}{2}s^{2}+\frac{1}{6}s^{3}.
\]
As $b\in(0,s_{0}]$ and $v$ is decreasing, we have
\begin{align*}
v(s_{0})   \leq v(b)&\leq k-\frac{b^{3}}{2}+\frac{1}{6}b^{3}\\
& =k-\frac{b^{3}}{3}\\
& <k-\frac{1}{3}\left(  \frac{k}{4}\right)  ^{3}\\
& =k\left (1-\frac{1}{192}k^{2}\right ).
\end{align*}
Therefore
\begin{align*}
v(s_{0})  & \leq
\left (1-\frac{1}{192}k^{2}\right )v(s_{\max}).
\end{align*}
As shown in the proof of the previous Lemma, $\left\vert v(s_{\min})\right\vert
<v\left(  s_{0}\right)  $, which now implies the result.

Now assume $v(s_{\max} ) >\pi$. The same argument used above implies that in order for $v$ to achieve a negative local minimum, $b > \frac \pi 4 $. Continuing with the same reasoning, we find 
\begin{align*}
v(s_{0})   \leq v(b)&\leq k-\frac{b^{3}}{3}\\
& <k-\frac{1}{3}\left(  \frac{\pi}{4}\right)  ^{3}\\
& =k-\frac{\pi^3}{192}.
\end{align*}
Therefore $|v( s _{\min} )| < v ( s_{\max} ) - \frac{\pi^3}{192}.$
\end{proof}

\begin{proposition}
\label{converges}Suppose a solution $v$ to \eqref{ODEsin} is bounded for
$s>0.$ Then $v(s)\rightarrow0.$
\end{proposition}

\begin{proof}
Let $v$ be a bounded solution. Either there are infinitely many critical
points, or finitely many. If there are finitely many (or none), the solution
must be monotone on an unbounded interval, so it converges to a
limit. This limit can't be anything other than $0$, otherwise the third
derivative would be uniformly positive or uniformly negative, contradicting boundedness.

So assume that $v$ has infinitely many critical points. Local maxima must be
positive and local minima must be negative, by the trifecta argument (Lemma \ref{trifecta}). By Lemma \ref{strictmin}, the maxima
must have strictly decreasing values and the minima are strictly increasing. We
claim the only value they can converge to is $0.$ Suppose that an infinite
sequence of local maxima has $k_{0}>0$ as a limit point. Then somewhere in
the sequence we have $v(s_{\max})< \min (\frac{1}{\delta(k_0)} k_{0}, k_0 + \frac {\pi^3}{192} )$. Lemma \ref{strictbounce} then
forces subsequent maxima below $k_{0}$ to obtain a contradiction.
\end{proof}

\subsection{Step 2}

\label{s33} Now we need to show that a solution to \eqref{ODE2} which is
bounded must actually be a solution to \eqref{ODE}. This is a modification of
the proof of Lemma \ref{strictmin}.

\begin{proposition}
\label{p36} A solution to \eqref{ODE2} with initial conditions \eqref{ICs}
that reaches $0$ or $\pi$ at some $s>0$ escapes to $-\infty$ or $\infty$,
respectively, never returning to $[0,\pi]$. In particular, any solution to
\eqref{ODE2} with these initial conditions which remains bounded is also a
solution to \eqref{ODE}.
\end{proposition}

\begin{proof}
We break this into three claims and then draw the conclusion.  

\begin{claim}
\label{c92}A solution to \eqref{ODE2} with initial conditions \eqref{ICs}
that reaches a critical point with critical value in $\left(  0,\pi\right)  $
and then exceeds $\pi$ and reaches a local maximum violates Lemma
\ref{strictmin}.
\end{claim}

\begin{claim}
\label{c93}A solution to \eqref{ODE2} with initial conditions \eqref{ICs}
that reaches $\pi$ without first finding a critical point will not achieve a
subsequent local maximum.  
\end{claim}

\begin{claim}
\label{c91}A solution to \eqref{ODE2} with initial conditions \eqref{ICs}
that reaches a critical point less than $\pi$, crosses $0$ and subsequently
achieves a local minumum violates Lemma \ref{strictmin}
\end{claim}

Proof of Claim \ref{c92}. Let $s_{1}$ be the first time $u$ crosses $\pi$ and
then let $s_{0}<s_{1}$ be the last critical point before this.  If
$u(s_{0})>$ $\pi/2$ this is a
trifecta;  no further critical points occur. Now if $u(s_{0})\leq
\pi/2$ we can translate Lemma \ref{strictmin}; this corresponds to solution
$v$ which achieves a minimum with a value in $(-\pi/2,0]$.  By the assumption, the subsequent
maximum would have a value strictly larger than $\pi/2$; this is
impossible by Lemma \ref{strictmin}.  

Proof of Claim \ref{c93}. As we are assuming that $u^{\prime\prime}(0)=0$, we can
repeat the proof of  Lemma \ref{strictmin}. Let $s_{1}$ be where $u$
crosses $\pi.$  Note that if $u^{\prime\prime}(s_{1})\geq0$, then $u$ will
enjoy the trifecta and will increase indefinitely afterwards.  So assuming
that $u$ is to achieve a critical value larger than $\pi$ leads us to
$u^{\prime\prime}(s_{1})<0.  $ After translating $v=u-\pi/2$ we get
\begin{align}
0 &  >u^{\prime\prime}(s_{1})-u^{\prime\prime}(0)\label{c192}\\
&  =v^{\prime\prime}(s_{1})-v^{\prime\prime}(0)=\int_{0}^{s_{1}}%
v^{\prime\prime\prime}(s)\:ds\nonumber\\
&  =\int_{0}^{s_{1}}g(v(s))\:ds=\int_{-\pi/2}^{\pi/2}g(v)\frac{ds}%
{dv}\:dv.\nonumber
\end{align}
Then we repeat the argument in the proof of  Lemma \ref{strictmin}; let
\[
\rho(v)=\frac{ds}{dv}=\frac{1}{v^{\prime}(s)}
\]
so that
\[
\frac{d}{dv}\rho(v)=\frac{-v^{\prime\prime}}{(v^{^{\prime}})^{3}}.
\]

As before, $v^{\prime\prime}(s)<0$ for all $s \in (0, s_1]$. 
It follows that
\[
\frac{d}{dv}\rho(v)=\frac{-v^{\prime\prime}}{(v^{^{\prime}})^{3}}\geq0\text{
for all }s\in\lbrack0,s_{1}].
\]
Integrating the increasing odd function $g$ against the increasing density function
$\rho$ we get
\[
\int_{-\pi/2}^{\pi/2}g\left(  v\right)  \rho(v)dv>0
\]
contradicting (\ref{c192}).  Thus $u$ cannot have a further critical value
larger than $\pi$, it must increase forever.

Proof of Claim \ref{c91}.  Translating $v=u-\pi/2$, look at the last critical
point $v$ achieves before crossing $- \pi/2$. If the critical value is negative, this represents a negative
trifecta. If the critical value is positive but less than $\pi/2$, the
subsequent minimum which we are assuming is less than $-\pi/2$ violates Lemma
\ref{strictmin} .  

The conclusion of the Lemma quickly follows:  Any solution reaching $\pi$ is
covered by Claim \ref{c92} or Claim \ref{c93}.  Any solution reaching 0
cannot have reached $\pi$ first by Claim \ref{c92} and Claim \ref{c93}, thus
Claim \ref{c91} applies. 
\end{proof}
\subsection{Step 3.} We now demonstrate that a bounded solution to \eqref{ODE2} with the initial conditions \eqref{ICs} exists. 
\begin{lemma}
\label{boundedsolutionexists}There exists a value $a$ such that the solution
to
\[
u^{\prime\prime\prime}=f(u)
\]
with
\begin{align*}
u(0)  &  =0,\\
u^{\prime}(0)  &  =a,\\
u^{\prime\prime}(0)  &  =0
\end{align*}
is bounded for all $s>0.$
\end{lemma}

\begin{proof}
Fix any $m\in\mathbb{N}.$ It's easy to check that there are values of $a$ such
that $u_{a}(m)>10$ and $u_{a}(m)<-10.$ By the order preservation property
Lemma \ref{order} and continuous dependence on initial conditions, there is a
set $I_{m}=\left[  a_{m},b_{m}\right]  $ such that
\begin{align*}
u_{a_{m}}(m)  &  =0,\\
u_{b_{m}}(m)  &  =\pi.
\end{align*}
By the order preservation property,
\[
I_{m}=\left\{  a\in\mathbb{R}^{+}\mid0\leq u_{a}(m)\leq\pi\right\}  .
\]
We claim that
\[
I_{m+1}\subset I_{m}.
\]
This follows immediately from Proposition \ref{p36}: If $u_{a}(m+1)$
$\in\lbrack0,\pi]$ then it cannot have left the region prior to that, so
$u_{a}(m)\in\left(  0,\pi\right)  .$

Now we conclude that
\[
I_{M}=\bigcap_{m\leq M}I_{m}%
\]
and we can take
\[
A=\bigcap_{m\in\mathbb{N}}I_{m}%
\]
which is the intersection of compact sets. Now the intersection of compact
sets either is nonempty or must be empty at a finite stage. By the shooting,
each $[a_{m},b_{m}]$ is nonempty. It follows that there must be at least one
point $a_{\infty}\in A.$ The solution shooting from $a_{\infty}$ stays bounded
for all time.
\end{proof}

\subsection{Step 4.} We may now take advantage of the odd initial conditions and the symmetry of the equation to reflect the solution, extending it to a bounded entire solution. 
\begin{corollary}
There exists $a$ such that the solution to the ODE 
$$ \theta^{\prime \prime \prime } = -\cos \theta$$ with initial
conditions
\begin{align*}
\theta(0)  &  =0,\\
\theta^{\prime}(0)  &  =a,\\
\theta^{\prime\prime}(0)  &  =0
\end{align*}
is bounded for all time and converges to $\pi/2$ as $s\rightarrow\infty$ and
$-\pi/2$ as $s\rightarrow-\infty.$
\end{corollary}

\begin{proof}
By Lemma
\ref{boundedsolutionexists}, there exists $a$ such that the initial conditions produce a solution $\theta(s)$ to (\ref{ODE2}) which is bounded for $s>0$. By Proposition \ref{p36}, $\theta(s)$ is a solution to
(\ref{ODE}), and it converges to $\frac\pi2 $ as $s \to\infty$ by Proposition
\ref{converges}. The initial conditions for $\theta$ given in \eqref{ICs} are odd.
The third order equation is even, so we may reflect the solution
$\theta(s)=-\theta (  -s)  $ for $s<0$ to obtain a solution for all $s.$ Reflecting the solution gives the desired properties in the
negative direction.
\end{proof}

\subsection{Step 5.} Finally, we can integrate the tangent directions to construct the profile curve. 
Given a tangent direction $\theta(s)$ we are left to solve the standard ODE
\begin{align*}
\frac{d}{ds}\gamma(s)  &  =T(s)=\left(  \cos\theta(s),\sin\theta(s)\right) ,\\
\gamma(0)  &  =\left(  0,0\right)
\end{align*}
which follows from basic ODE theory. Given that $\theta$ solves \eqref{ODE},
the result is a curve whose flow under \eqref{equation1} results in translation.

The solution $\gamma$ is clearly an immersion, the tangent vector is
nonvanishing by construction. To see that it must be a proper immersion, note
that because $\theta(s)\rightarrow\pi/2$ as $s\rightarrow\infty$ we can find a
value $\sigma_{0}$ such that $\theta(s)\in\left(  \frac{\pi}{4},\frac{3\pi}%
{4}\right)  $ for $s\geq\sigma_{0}.$ Letting $\gamma(s)=\left(
x(s),y(s)\right)  $ we have
\[
y\left(  s\right)  -y(\sigma_{0})=\int_{\sigma_{0}}^{s}\sin\theta
(s)\:ds\geq\frac{\sqrt{2}}{2}(s_{1}-\sigma_{0})\text{ for }s>\sigma_{0}.
\]
Repeating this argument for negative values of $s,$ we see the immersion must
be proper. This proves Theorem \ref{theorem}.

\
\bibliographystyle{amsalpha}
\bibliography{hflow}

\end{document}